\documentclass[openright, 12pt]{article}
\usepackage{amsmath,amsfonts}
\usepackage{amsthm}
\newtheorem{thm}{Theorem}

\newtheorem{lem}[thm]{Lemma}
\newtheorem{rem}[thm]{Remark}

\newtheorem{defn}[thm]{Definition}
\newtheorem{con}[thm]{Conjecture}
\usepackage{float} %NO SE PARA QUE SE USA
\usepackage{multirow} %NO SE PARA QUE SE USA
\usepackage[dvips]{graphicx}
\usepackage{psfrag}
\usepackage{enumerate}

\usepackage{amssymb}
\usepackage[latin1]{inputenc}

\begin{document}

\begin{center}
\begin{Large}
Tori Can't Collapse to an Interval 
\end{Large}

\begin{normalsize}
Sergio Zamora

sxz38@psu.edu
\end{normalsize}

\end{center}

%----------Author 1
%\author[Sergio Zamora]{Sergio Zamora}

%\address{%
%Penn State University\\
%018 McAllister Bldg\\
%University Park, PA\\
%16801, USA }

%\email{sxz38@psu.edu}

%\subjclass{53C23}

%\keywords{Alexandrov geometry, collapse, sectional curvature bounds}

%\date{July 28, 2020}

\begin{abstract}
Here we prove that under a lower sectional curvature bound, a sequence of Riemannian manifolds diffeomorphic to the standard $m$-dimensional torus cannot converge in the Gromov--Hausdorff sense to a closed interval. 

The proof is done by contradiction by analyzing suitable covers of a contradicting sequence, obtained from the Burago--Gromov--Perelman generalization of the Yamaguchi fibration theorem.
\end{abstract}

%%% ----------------------------------------------------------------------
%\maketitle
%%% ----------------------------------------------------------------------
%\tableofcontents
\section{Introduction}

The study of Riemannian manifolds with sectional curvature bounded below naturally leads to the study of Alexandrov spaces in part due to the following well known results.

\begin{thm}\label{Compactness}
(\cite{PG}, Theorem 5.3). Let $X_n$ be a sequence of closed $m$-dimensional Riemannian manifolds with sectional curvature $\geq c$ and diameter $\leq D$. Then there is a subsequence that converges in the Gromov--Hausdorff sense to a compact space $X$. 
\end{thm}

\begin{thm} \label{Alex}
(\cite{BGP}, Section 8). Let $X_n$ be a sequence of closed $m$-dimensional Riemannian manifolds with sectional curvature $\geq c$. If the sequence $X_n$ converges in the Gromov--Hausdorff sense to a compact space $X$, then $X$ is an $\ell$-dimensional Alexandrov space of curvature $\geq c$ with $\ell \leq m$. 
\end{thm}

In this situation, results by Perelman and Yamaguchi show that in many cases, the topology of the limit is closely tied to the topology of the sequence.

\begin{thm} \label{PerelStab}
\cite{Per}, \cite{KaPer}. Let $X_n$ be a sequence of closed $m$-dimensional Riemannian manifolds with sectional curvature $\geq c$. If the sequence $X_n$ converges in the Gromov--Hausdorff sense to a compact $m$-dimensional Alexandrov space of curvature $\geq c$, then there is a sequence $f_n : X_n \to X$ of Gromov--Hausdorff approximations such that $f_n$ is a homeomorphism for large enough $n$.
\end{thm}

\begin{thm} \label{Fibration}
(\cite{Yam}, Main Theorem). Let $X_n$ be a sequence of closed $m$-dimensional Riemannian manifolds with sectional curvature $\geq c$. If the sequence converges in the Gromov--Hausdorff sense to a closed Riemannian manifold $X$, then there is a sequence of Gromov--Hausdorff approximations $f_n : X_n \to X$ such that $f_n$ is a locally trivial fibration for large enough $n$.
\end{thm}

Even with these two powerful theorems, collapsing under a lower curvature bound is still far from being well understood, specially when the limit space has singularities or boundary. At the local level, Vitali Kapovitch has successfully studied the behaviour of collapse.

\begin{thm}
\cite{Kap}. Let $X_n$ be a sequence of closed $m$-dimensional Riemannian manifolds with sectional curvature $\geq c$ that converges in the Gromov--Hausdorff sense to a compact Alexandrov space $X$ of curvature $\geq c$. Then for any $x_0 \in X$, any sequence of Gromov--Hausdorff approximations $f_n : X_n \to X$, and any sequence of points $x_n \in f_n^{-1}(x_0)$, there is an $r_0 (x_0) > 0$ such that for large enough $n$, the closed ball $B_{r_0}(x_n) $ is a manifold with boundary, simply homotopic to a finite CW-complex of dimension $\leq m-\ell$. 
\end{thm}

At the global level, Mikhail Katz recently proved that the $2$-dimensional torus cannot collapse to a segment.

\begin{thm}\label{Katz}
\cite{Katz}. Let $g_n$ be a sequence of Riemannian metrics of sectional curvature $\geq -1$ in the $2$-dimensional torus $M$. Then it cannot happen that the sequence $(M, g_n)$ converges in the Gromov--Hausdorff sense to an interval $[0,L]$.
\end{thm}

The goal of this note is to prove the following generalization of Theorem \ref{Katz}.

\begin{thm}\label{main}
Let $g_n$ be a sequence of Riemannian metrics of sectional curvature $\geq -1$ in the $m$-dimensional torus $M$. Then it cannot happen that the sequence $(M, g_n)$ converges in the Gromov--Hausdorff sense to an interval $[0,L]$.
\end{thm}

\begin{rem}
Let $\Phi_n$ be the group of isometries of $\mathbb{C}$ generated by $z\to z+2i$ and $z \to \overline{z}+\frac{1}{n}$. The quotient $W_n = \mathbb{C}/\Phi_n$ is a flat Klein bottle and the sequence $W_n$ converges in the Gromov--Hausdorff sense to $[0,1]$ (see Figure \ref{Klein}), so Theorem \ref{main} is false if one replaces the $m$-dimensional torus by the Klein bottle.   
\end{rem}

\begin{figure}
\centering
\includegraphics[scale=0.7]{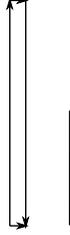}
\caption{Flat Klein bottles can converge to an interval}\label{Klein}
\end{figure}

Theorem \ref{main} represents a step towards the following conjecture. Theorem \ref{PerelStab} implies the case $\ell = m$, and Theorem \ref{main} the case $\ell =1$.

\begin{con}\label{tortotor}
Let $g_n$ be a sequence of Riemannian metrics of sectional curvature $\geq -1$ in the $m$-dimensional torus $M$ such that the sequence $(M, g_n)$ converges to a compact $\ell$-dimensional Alexandrov space $X$. Then $X$ is homeomorphic to an $\ell$-dimensional torus.
\end{con}

\begin{rem}
Gromov and Lawson showed in \cite{GL} that if a Riemannian manifold diffeomorphic to the $m$-dimensional torus has scalar curvature $\geq 0$, then it is flat. The Mahler compactness Theorem asserts that if $g_n$ is a sequence of flat Riemannian metrics in the $m$-dimensional torus $M$, such that the sequence $(M, g_n)$ converges in the Gromov--Hausdorff sense to a compact space $X$, then $X$ is a flat torus (see \cite{Cass}, p.137). Therefore Conjecture \ref{tortotor} is known to be true if we replace sectional curvature $\geq -1$ by scalar curvature $\geq 0$.
\end{rem}

The structure of this note is as follows: In section 2, we give a consequence of our main theorem. In section 3 we give the necessary definitions and preliminaries. In section 4 we give the proof of Theorem \ref{main}.

The author would like to thank Raquel Perales and Anton Petrunin for stimulating his interest in Conjecture \ref{tortotor}, and an anonymous reviewer whose comments improved the exposition of this paper.

\section{Flat Manifolds}

A little bit more can be said about manifolds admitting flat metrics. Recall Bieberbach Theorem and the definition of holonomy group. An elegant proof can be found in \cite{Cha}.

\begin{thm}
 Let $M$ be a flat closed $m$-dimensional manifold. Then its fundamental group fits in an exact sequence 
\begin{center}
$ 0 \to \mathbb{Z}^m \to \pi_1(M) \to H_M \to 0.   $
\end{center}

The group  $\mathbb{Z}^m$ is the only maximal abelian normal subgroup of $\pi_1 (M)$. The group $H_M$ is finite and it is called the \textit{holonomy group of }$M$. The cover associated to $\mathbb{Z}^m \leq \pi_1(M)$ is a flat torus.
\end{thm}

\begin{thm}\label{Flat}
 Let $M$ be a closed $m$-dimensional manifold that admits a flat metric. If there is a sequence $g_n$ of Riemannian metrics with $sec(M, g_n) \geq -1$ such that $(M, g_n) $ converges to an interval $[0,L]$, then the holonomy group $H_M$ has a subgroup of index 2.
\end{thm}

For the proof of Theorem \ref{Flat}, we will need the following elementary result.

\begin{lem}\label{Equiv}
(\cite{GPoly}, Section 6)  Let $\Gamma$ be a finite group and $X_n$ a sequence of compact metric spaces converging in the Gromov--Hausdorff sense to the space $X$. Assume we have a sequence of isometric group actions $\Gamma \to Iso (X_n)$. Then there is an isometric group action $\Gamma \to Iso (X)$ such that a subsequence of $X_n / \Gamma$ converges in the Gromov--Hausdorff sense to $X / \Gamma $.  
\end{lem}

\begin{proof}[Proof of Theorem \ref{Flat}]
Let $X_n$ be the torus metric cover of $(M, g_n)$ with $X_n / H_M = (M, g_n) $. The diameter of $X_n$ is bounded above by $2L \vert H_M \vert $ for large enough $n$, hence by Theorems \ref{Compactness} and \ref{Alex}, $X_n$ converges up to subsequence to an Alexandrov space $X$ of dimension $\leq m$. By Lemma \ref{Equiv}, there is an isometric group action $H_M \to Iso (X)$ such that $X/ H_M = [0,L]$. The quotient of a compact finite dimensional Alexandrov space by a finite group is another compact Alexandrov space with the same dimension. This means that $X$ is a compact $1$-dimensional Alexandrov space, so it is a circle or a closed interval. By Theorem \ref{main}, $X$ is a circle, and the action $ H_M \to Iso(X)$ is either cyclic or dyhedral. The quotient of $X$ by a dihedral group is an interval, and the quotient of $X$ by a cyclic group is a shorter circle. Since $X/H_M = [0,L]$, the image of $H_M$ in $Iso(X)$ is a dihedral group, which has a subgroup of index 2. 
\end{proof}

Theorem \ref{Flat} implies in particular that if the holonomy group $H_M$ is simple, or has odd order, then $M$ cannot collapse to an interval under a lower sectional curvature bound. The following Theorem by Auslander and Kuranishi tells us the relevance of Theorem \ref{Flat}.

\begin{thm}
(\cite{AK}, Theorem 3). Let $H$ be a finite group. Then there is a flat manifold $M$ with $H_M = H$.
\end{thm}

\section{Prerequisites}

\subsection{Gromov--Hausdorff distance}

Gromov--Hausdorff distance was introduced to quantitatively measure how far two metric spaces are from being isometric.
\begin{defn}
Let $A,B$ be two metric spaces. The Gromov--Hausdorff distance $d_{GH}(A,B)$ between $A$ and $B$ is defined as 
$$d_{GH}(A,B) : = \inf_{\varphi, \psi} d_H( \varphi (A), \psi (B) )  ,   $$
 where $d_H$ denotes the Hausdorff distance, and the infimum is taken over all isometric embeddings $\varphi : A \to C$, $\psi : B \to C$ into a common metric space $C$.
\end{defn}

We refer to (\cite{BBI}, Chapter 7) for the basic theory about Gromov--Hausdorff distance, including the following equivalence of convergence with respect to such metric.

\begin{thm}
Let $X_n$ be a sequence of compact metric spaces. The sequence $X_n$ converges in the Gromov--Hausdorff sense to a compact metric space $X$ if and only if there is a sequence of maps $f_n : X_n \to X$ such that
$$ \lim_{n \to \infty} \sup_{x,y \in X_n} \vert d(f_n (x), f_n (y)) - d(x,y) \vert =0,   $$
and
$$   \lim_{n \to \infty} \sup_{x \in X}  \inf_{y \in X_n} d(x, f_n ( y)   ) =0  .$$
A sequence of functions $f_n$ satisfying the above properties are called Gromov--Hausdorff approximations. 
\end{thm}

\subsection{Yamaguchi--Burago--Gromov--Perelman Fibration}

In a compact Alexandrov space $X$ of dimension $\ell$, one can quantify how degenerate a point $p \in X$ is by studying the Gromov--Hausdorff distance between its space of directions $\Sigma _pX$ and the standard sphere $\mathbb{S}^{\ell -1}$.  To construct $\Sigma_pX$, one needs to put a metric on the set of geodesics in $X$ emanating from $p$. Given two minimizing geodesics $\gamma_i : [0, \delta_i] \to X$ parametrized by arc length with $\gamma_i(0)=p$ for $i=1,2$, we define the angular distance between them as

$$  d (\gamma_1, \gamma_2 ) : = \lim _{s,t \to 0^+} \cos^{-1}  \left(  \dfrac{ d( \gamma_1(s), \gamma_2 (t)  )^2 - s^2 - t^2   }{st}    \right)   \in [0, \pi]  .        $$

The set of all minimizing geodesics in $X$ starting from $p$ equipped with the angular metric form a semi-metric space $S_pX$. The space $\Sigma_pX$ is defined as the metric completion of the metric space asociated to $S_pX$. 

 For $\delta >0$, we say that a point $p$ is $\delta$-regular if  $d_{GH} (\Sigma_pX , \mathbb{S}^{\ell -1}) < \delta$. The set of $\delta$-regular points $U_{\delta}(X) \subset X $ form an open dense set, and for small enough $\delta$, they form an $\ell$-dimensional (topological) manifold. Burago, Gromov, and Perelman noticed that Theorem \ref{Fibration} has a version for when $X$ is singular.

\begin{thm}\label{SFibration}
(\cite{BGP}, Section 9). For small enough $\delta (m, c)$ the following holds. Let $X_n$ be a sequence of closed $m$-dimensional Riemannian manifolds with sectional curvature $\geq c$ converging in the Gromov--Hausdorff sense to a compact space $X$. Then for any compact $K \subset U_{\delta}(X)$ there is a sequence of Gromov--Hausdorff approximations $f_n : X_n \to X$ such that for large enough $n$, $f_n \vert_{f_n^{-1}(K)}$ is continuous and morover, it is a locally trivial fibration with fiber $F_n$, a compact almost nonnegatively curved manifold in the generalized sense of dimension $m-\ell$ (ANNCGS($m-\ell$)) (see \cite{KPT}, Definition 1.4.1). 
\end{thm}

Since the definition of an almost non-negatively curved manifold in the generalized sense is technical and we will not use it, we will omit it. The only result we will need regarding such manifolds is the following.

\begin{thm}\label{YamBetti}
(\cite{Yam}, Pinching Theorem). The first Betti number of  an $ANNCGS(d)$ is $\leq d$. 
\end{thm}

\subsection{Comparison Geometry}

We will use two elementary facts about comparison geometry. The first one is called the four point Alexandrov condition.

\begin{lem}\label{AlexDBGP}
(\cite{BGP}, Section 2). Let $N$ be an $m$-dimensional Riemannian manifold with sectional curvature $\geq c$ and $\mathbb{M}^{m}(c)$ be the simply connected complete $m$-dimensional Riemannian manifold of constant curvature $c$. For distinct points $p, a_1,a_2,a_3 \in N$ and $i\in \{ 1,2,3\} $, we set $a_4=a_1$ and call $\theta_i$ the angle at $\tilde{p}$ of a triangle $\tilde{a}_i\tilde{p}\tilde{a}_{i+1}$ in $ \mathbb{M}^m(c)$ with $d(\tilde{p}, \tilde{a}_i) = d( p, a_i) $, $d(\tilde{p}, \tilde{a}_{i+1}) = d( p, a_{i+1}) $, $d(\tilde{a}_i, \tilde{a}_{i+1}) = d( a_i, a_{i+1}) $. Then
$$  \theta_1 + \theta_2 + \theta _3 \leq 2 \pi. $$
This condition is called the Alexandrov condition for the quadruple $(p; a_1, a_2, a_3)$.
\end{lem}

Another ingredient is the Bishop--Gromov inequality.

\begin{thm}\label{BGI}
 (\cite{BC}, p. 253). Let $Z$ be an $m$-dimensional Alexandrov space of curvature $\geq c$, and $\mathbb{M}^{m}(c)$ be the simply connected complete $m$-dimensional Riemannian manifold of constant curvature $c$. Then for $0 < r< R$, $p \in Z$, $q \in \mathbb{M}^m(c)$, we have

$$ \dfrac{Vol (B_R(p)) }{Vol(B_r(p))} \leq \dfrac{Vol(B_R(q))}{Vol (B _r(q))  }  .  $$ 
 
\end{thm}

\subsection{Miscelaneus results}

\begin{lem}\label{Shortgene}
(\cite{PG}, Proposition 5.28) Let $Z$ be a compact semilocally simply connected length space, $z_0 \in Z$, $\eta >0$, and $r = \sup _{ z \in Z} d(z,z_0)$. Then $\pi_1(Z,z_0)$ is generated by the loops of length $\leq 2r+\eta$.
\end{lem}

We will also use a simple version of Gromov's systolic inequality.

\begin{thm}\label{Sys}
(\cite{FRM}, Section 1). Let $N$ be a  smooth closed aspherical manifold, and $g_n$ a sequence of Riemannian metrics on $N$ such that the volumes of the spaces $Z_n = (N, g_n)$ go to $0$ as $n \to \infty$. Then there is a sequence of noncontractible loops $\gamma_n : \mathbb{S}^1 \to Z_n$ with lengths going to $0$ as $n \to \infty$
\end{thm}

\section{Proof of Theorem \ref{main}}

Assume by contradiction, that there is a sequence $X_n = (M, g_n)$ as in Theorem \ref{main} converging to an interval $[0,L]$. Applied to the limit space $[0,L]$, Theorem \ref{SFibration} takes the following form. 

\begin{lem}\label{MFibration}
For any $\varepsilon >0$, and large enough $n (\varepsilon )$, there are continuous Gromov--Hausdorff approximations $f_n : X_n \to [0,L]$ such that $f_n^{-1}([\varepsilon , L- \varepsilon ])$ is homeomorphic to the product $[\varepsilon , L- \varepsilon ] \times F_n$, with $F_n$ an ANNCGS($m-1$), and  $f_n \vert_{f_n^{-1}([\varepsilon , L- \varepsilon ])}$ being the projection onto the first factor.
\end{lem}

\subsection{2-dimensional case}

Proving Theorem \ref{main} for $m=2$ is easier and gives an idea on how to get the general case. Fix a small $\varepsilon  $ (say, $\varepsilon = L/100$) and use Lemma \ref{MFibration}. We see that the fibers $F_n$ are homeomorphic to $\mathbb{S}^1$ (the only compact 1-dimensional manifold), exhibiting $X_n$ as the connected sum of two surfaces $S_1 \# S_2$ (see Figure \ref{Cosum}). Since the 2-dimensional torus is undecomposable, one of the surfaces, say $S_1$, is homeomorphic to $\mathbb{S}^2$. This would imply that $Y_n := f_n ^{-1} ([0, L- \varepsilon])$ is homeomorphic to a disk, meaning that the inclusion $Y_n \to X_n$ is trivial at the level of fundamental groups. Therefore, when we take the universal covering $\tilde{X}_n \to X_n$, the preimage of $Y_n$ consists of disjoint copies of $Y_n$ (one for each element of $ \pi_1(X_n) =  \mathbb{Z}^2$).

\begin{figure}
\centering
\includegraphics[scale=1.14]{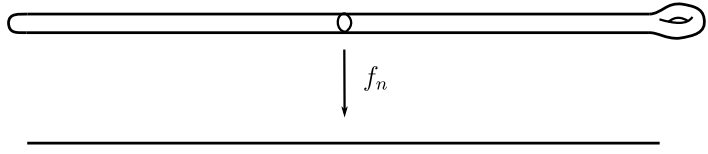}
\caption{The Fibration Theorem gives us a decomposition $X_n = S_1 \# S_2$}\label{Cosum}
\end{figure}

The sequence $X_n$ collapses to a lower dimensional object, so the volume of $X_n$ goes to $0$ as $n\to \infty$.  Since the torus is aspherical, by Theorem \ref{Sys} for any $C \in \mathbb{N}$, and large enough $n(C)$, there are non contractible loops $\gamma_n : \mathbb{S}^1 \to X_n$ of length $\leq L/C$. Since $Y_n$ is contractible and $\gamma_n$ is noncontractible, we have $\gamma_n (\mathbb{S}^1)   \backslash Y_n \neq \emptyset $. Let $x_n = \gamma_n (1)$, and $\tilde{x}_n $ one of its preimages in $\tilde{X}_n$. Since $\mathbb{Z}^2 = \pi_1 (X_n)$ has no torsion, there are at least $C/3$ elements of the orbit of $\tilde{x}_n$ in the ball $B_{L/2}(\tilde{x}_n)$.

Let $q_n \in f_n ^{-1}([0,\varepsilon ])$, and $\tilde{q}_n \in \tilde{X}_n$ its preimage closest to $\tilde{x}_n$. The ball $B_{L-3 \varepsilon}(\tilde{q}_n)$ is isometric to  $B_{L-3 \varepsilon}(q_n)$. However, the ball $B_{3L}(\tilde{q}_n)$ contains at least $C/3$ disjoint isometric copies of $B_{L-3 \varepsilon}(q_n)$.

By Theorem \ref{BGI}, applied to $Z= \tilde{X}_n$, $c=-1$, $p= \tilde{q}_n$, $r= L -3 \varepsilon$,  $R = 3L$, we get for any $q \in \mathbb{M}^m(-1)$,

$$  \dfrac{C}{3} \leq  \dfrac{Vol (B_{3L}(\tilde{q}_n)) }{Vol(B_{L-3\varepsilon }( \tilde{q}_n ))} \leq \dfrac{Vol(B_{3L}(q))}{Vol (B _{L-3 \varepsilon }(q))  }  .    $$

The right hand side only depends on $m, L, \varepsilon$, so the above inequality cannot hold if $C$ is large enough, which is a contradiction.

\subsection{General case.}

Fix a small $\varepsilon $ (to be chosen later) and use Lemma \ref{MFibration}. By Theorem \ref{YamBetti} we have that for large $n$, the image of the morphism $i_{\ast}: \pi_1 (F_n )\to \pi_1 (X_n)$ induced by the inclusion $i : F_n \to X_n$ has corank at least 1. Let $\tilde{X}_n$ be the cover of $X_n$ with Galois group $\Gamma_n := \pi_1 (X_n)/ i_{\ast}\pi_1(F_n)$.  Observe that by construction, the preimage of $ f_n^{-1}([\varepsilon , L- \varepsilon]) $ in $\tilde{X}_n$ consists of disjoint copies of itself.

Let $p_n$ be a point in $f_n^{-1}(L/2)$ and $\tilde{p}_n$ a lift in $\tilde{X}_n$. Let $S$ be the set of loops in $X_n$ based at $p_n$ of length $\leq L + 10 \varepsilon $. By Lemma \ref{Shortgene}, for large enough $n$, $S$ generates $\pi_1 (X_n, p_n)$. The elements of $S$ whose image is contained in $f_n^{-1}([\varepsilon , L- \varepsilon])$ are homotopic to elements of $i_{\ast} \pi_1(F_n)$ and lift to loops in $\tilde{X}_n$. Let $S^{\prime}$ be the subset of $S$ not homotopic to elements in $i_{\ast}\pi_1(F_n)$. $S^{\prime}$ generates $\Gamma_n$ and consists of loops that go to one of $f_{n}^{-1}([0, \varepsilon])$ or $f_n^{-1}([L- \varepsilon, L])$, but not both. We will call them Type I or Type II depending on whether they visit $f_{n}^{-1}([0, \varepsilon])$ or $f_n^{-1}([L- \varepsilon, L])$. 

First assume that there are no loops of Type I. This would mean that the inclusion
$$  j : f_n^{-1} (   [0, L- \varepsilon ]) \to X_n   $$
induces a map at the level of fundamental groups such that 
$$j_{\ast} (  \pi_1 (  f_n^{-1} (   [0, L- \varepsilon ])   )  ) \subset i_{\ast} (\pi_1 (F_n)).$$
This implies that the preimage of $ f_n^{-1}([0, L- \varepsilon ])$ in $\tilde{X}_n$ consists of infinitely many disjoint copies of itself. Also, since $\Gamma_n$ is abelian of positive rank, any set of generators contains an element of infinite order. Then there is a loop of Type II of infinite order and we can conclude identically as in the 2-dimensional case.

Now assume that there are two loops $\alpha$, $\beta$ of Type I not equivalent in $\Gamma_n$. This means that they lift as paths $\tilde{\alpha}$, $\tilde{\beta}$ in $\tilde{X}_n$ with startpoint $\tilde{p}_n$, but distinct endpoints $a_n$, $b_n$, respectively. Letting $q_n$ be an approximate midpoint of $a_n$ and $\tilde{p}_n$ in the image of $\tilde{\alpha}$ we see that 

\begin{center}
$d(\tilde{p}_n,a_n )  \approx d(\tilde{p}_n, b_n) \approx d(a_n, b_n) \approx L  $

 $ d(q_n,\tilde{p}_n) \approx d(q_n, a_n) \approx d(q_n , b_n) \approx L/2,  $
\end{center}
 
where the error in the above  approximations is of the order of $\varepsilon$. This violates the Alexandrov condition for the quadruple $(q_n ; \tilde{p}_n,a_n, b_n )$  if $\varepsilon (L)$ was chosen small enough (see Figure \ref{AlexD}).

\begin{figure}
\centering
\includegraphics[scale=0.6]{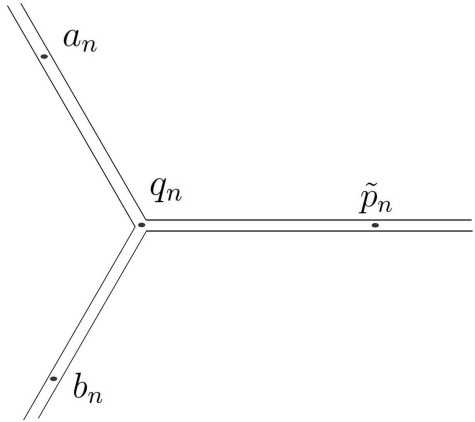}
\caption{The configuration $(q_n; \tilde{p}_n ,a_n, b_n) $ violates the Alexandrov condition}\label{AlexD}
\end{figure}

With this, we see that in $S^{\prime}$ there is exactly one loop of Type I and one loop of Type II modulo $i_{\ast}\pi_1(F_n)$. Observe that the inverse in $\Gamma_n$ of the loop of Type I is also a loop of Type I, but there is only one loop of Type I in $\Gamma_n$, so it is its own inverse, same for the loop of Type II. But $\Gamma_n$ is abelian of positive rank, so it cannot be generated by two elements of order 2.

%------------------------------------------------------------------------

\subsection*{Acknowledgment}

On behalf of all authors, the corresponding author states that there is no conflict of interest.

\end{document}